\documentclass[11pt]{article}
\usepackage{latexsym}
\usepackage{amsthm}
\usepackage{amssymb}
\usepackage{amsmath}
\usepackage{rotating}
\usepackage{multirow}
\usepackage{mathrsfs}
\usepackage{wasysym}
\usepackage{esint}
\usepackage{rawfonts}
\input{prepictex}
\input{pictex}
\input{postpictex}
\DeclareMathOperator{\End}{End}
\DeclareMathOperator{\Diff}{\zap{Diff}}
\usepackage[OT2,OT1]{fontenc}
\def\cyr{%
\renewcommand\rmdefault{wncyr}%
\renewcommand\sfdefault{wncyss}%
\renewcommand\encodingdefault{OT2}%
\normalfont
\selectfont}
\DeclareMathAlphabet{\zap}{OT1}{pzc}{m}{it}
\DeclareTextFontCommand{\textcyr}{\cyr}
\def\be{\begin{equation}}
\def\ee{\end{equation}}
\def\bea{\begin{eqnarray*}}
\def\eea{\end{eqnarray*}}

\def\CC{\mathbb C}

\newtheorem{main}{Theorem}

\newtheorem{thm}{Theorem}
\newtheorem{lem}{Lemma}
\newtheorem{prop}{Proposition}
\newtheorem{cor}{Corollary}

\def\RR{{\mathbb R}}
\def\CP{{\mathbb C \mathbb P}}
\begin{document}

\title{Einstein Metrics, Harmonic Forms,\\ and Symplectic Four-Manifolds}

\author{Claude LeBrun\\Stony Brook University}

\date{}
\maketitle

\begin{abstract}
If $M$ is the underlying smooth oriented $4$-manifold of a Del Pezzo surface, we  consider the 
set of Riemannian metrics $h$ on $M$ such that $W^+(\omega , \omega )> 0$,
where $W^+$ is the self-dual Weyl curvature of $h$, and $\omega$ is a non-trivial 
 self-dual harmonic $2$-form on $(M,h)$. While  this open region in the space of
Riemannian metrics  contains all the known Einstein metrics on $M$, we show that it contains  no others. Consequently, it 
 contributes exactly one connected component to the moduli space of Einstein metrics on $M$. 
\end{abstract}

\section{Introduction}

Given a smooth compact $4$-manifold $M$, one would like to completely understand the  moduli
space $\mathscr{E}(M)$ of the Einstein metrics  it carries. Recall that  an {\em Einstein metric} \cite{bes} means a Riemannian metric $h$ 
which has  constant Ricci curvature, in the sense that it solves the Einstein equation
$$r= \lambda h$$
where $r$ is the Ricci tensor of $h$ and  $\lambda$ is a real number, called the {\em Einstein constant}
of $h$. The {\em moduli space} $\mathscr{E}(M)$ of Einstein metrics on $M$ is by definition the quotient of the set of Einstein metrics by 
the action of the group $\Diff (M) \times \RR^\times$ of self-diffeomorphisms and constant rescalings. For simplicity, 
we may  give $\mathscr{E}(M)$ the quotient topology 
induced by  the $C^{\infty}$-topology on the space of smooth metric tensors; however,  
it is worth noting that, for reasons of  elliptic regularity \cite{det-kaz},   this coincides \cite{andmod} with the metric topology induced by the Gromov-Hausdorff distance
between unit-volume Einstein metrics.

While our understanding of this problem is rather limited for general $4$-manifolds, there are specific cases where our knowledge 
is quite complete. In particular, if $M$ is the $4$-torus, or $K3$, or a compact real or complex-hyperbolic $4$-manifold, 
the Einstein moduli space $\mathscr{E}(M)$  is known  to be {\em connected} \cite{bes,bcg,lmo}. This should be contrasted with  the pattern
that predominates   in 
higher dimensions, where Einstein moduli spaces  are typically  disconnected, and  indeed often have infinitely many 
connected components \cite{bohm,bgk,cvc}.

This article will explore  related uniqueness 
   questions for 
   Einstein metrics on  the small but important  class of smooth compact  $4$-manifolds  that arise  as {\em Del Pezzo surfaces}. 
  These 
$4$-manifolds are characterized \cite{chenlebweb} by  two properties:   they admit Einstein metrics with $\lambda > 0$, and they also admit symplectic structures. 
 Up to diffeomorphism, there are
exactly ten such manifolds, namely  $S^2 \times S^2$ and the nine connected sums $\CP_2\# m\overline{\CP}_2$,  $m = 0, 1, \ldots, 8$. 
The known Einstein metrics on these spaces all have $\lambda > 0$, and our objective here is to completely characterize these known Einstein metrics by 
a curvature condition. To this end, notice that, with their standard orientations, each of these $4$-manifolds $M$ has  $b_+(M)=1$. This is equivalent to saying that, 
for any Riemann metric $h$ on any of these compact oriented $4$-manifolds, there is, up to an overall multiplicative constant, a unique non-trivial self-dual harmonic 
$2$-form $\omega$. We can therefore consider those Riemannian metrics $h$ on $M$ which satisfy the curvature inequality
\begin{equation}
\label{rosetta}
W^+(\omega , \omega ) > 0
\end{equation}
at every point of $M$, where $W^+$ is the self-dual Weyl tensor of $h$. Note that, whenever  $b_+(M)=1$, this  condition only depends on the metric $h$, since 
the  harmonic form $\omega$ is then uniquely determined up to a non-zero  multiplicative constant. 
Our first main result is the following: 

\begin{main} 
\label{gemstone}
Let $(M,h)$ be a compact oriented Einstein $4$-manifold with $b_+=1$, and suppose that  condition  \eqref{rosetta} holds 
at every point of $M$. Then $M$ is diffeomorphic to a Del Pezzo surface, in such a way that   $h$ becomes 
\begin{itemize}
\item a K\"ahler-Einstein metric with $\lambda > 0$; or 
\item  a constant multiple of the Page metric on $\mathbb{CP}_2\# \overline{\mathbb{CP}}_2$; or 
\item a constant multiple of 
 the Chen-LeBrun-Weber metric on  $\mathbb{CP}_2\# 2\overline{\mathbb{CP}}_2$.
 \end{itemize}
 Conversely, every metric on this list satisfies \eqref{rosetta} at every point. 
 \end{main}
 
 The proof of this result is given in \S \ref{torso} below, and proceeds by proving that the given Einstein metric
 must be conformally K\"ahler.  Our method makes strong use of the fact that the second Bianchi identity 
 implies that the self-dual Weyl curvature $W^+$ of an oriented 
 $4$-dimensional Einstein manifold  is  {\em harmonic}  as a bundle-valued $2$-form. In fact, 
the proof does not really require the assumption that $b_+(M)=1$; it suffices to  assume that there is 
 a harmonic self-dual $2$-form $\omega$ on $(M,h)$ such that \eqref{rosetta} holds at every point. 
 
 While 
 the inequality \eqref{rosetta}
 may have a somewhat unfamiliar flavor, it 
  is interestingly related to the positivity of scalar curvature. Indeed, any harmonic self-dual $2$-form $\omega$
 satisfies the Weitzenb\"ock formula 
 $$\frac{1}{2}\Delta |\omega|^2 + |\nabla \omega |^2  + \frac{s}{3} |\omega |^2 = 2 W^+(\omega , \omega ), $$
and it therefore follows that any metric of  scalar curvature $s > 0$ must at least satisfy \eqref{rosetta} at {\em some} points of $M$. Thus, while 
Theorem \ref{gemstone} does not  provide  a complete classification of  $\lambda > 0$ Einstein metrics on Del Pezzo surfaces, it does represent  a  step
in that direction. 

However,  our method requires  \eqref{rosetta} to hold {\em everywhere}, rather than 
just at certain points. This is a strong condition, because it  guarantees that  the closed self-dual  $2$-form $\omega$ is nowhere zero,  and
therefore  implies  that  $(M, \omega )$ is a symplectic manifold. If 
 $h$  is a Riemannian metric on a smooth compact oriented $4$-manifold $M$ with $b_+(M)\neq 0$, 
 we will thus say that $h$ is of {\em symplectic type} if there is a self-dual harmonic $2$-form 
 on $(M,h)$ such that $\omega \neq 0$ at every point of $M$. 
 This is actually a conformally invariant condition; if $h$ is of symplectic type, and if
 $u$ is a smooth positive function, then $uh$ is also of symplectic type.  
 For this reason, it is also natural to say  that the conformal class $[h]$ is of symplectic type if 
 there is a self-dual harmonic $2$-form on $(M, [h])$ which is everywhere non-zero. 
 This is an {\em open condition}
on $[h]$, in the sense that  the set of conformal classes of  symplectic type is automatically open \cite{lcp2} in the $C^2$ topology.

 Condition \eqref{rosetta} is also conformally invariant. Namely, if we replace  
 $h$ with $u h$ for
some positive function $u$, then $W^+(\omega, \omega)$ is replaced with $u^{-3}W^+(\omega, \omega)$,
thereby leaving the  the sign of $W^+(\omega, \omega)$ unaltered at any given point.  We will henceforth say that  the conformal class $[h]$  is  of 
{\em positive symplectic type} if, for some choice of $h$-compatible self-dual  harmonic $2$-form $\omega$, condition 
\eqref{rosetta} holds everywhere on $M$. This obviously implies that   $\omega \neq 0$ everywhere, so positive symplectic type implies  symplectic type. 
The condition of positive symplectic type is once again  open in the $C^2$ topology. 

With these concepts in place, we are now ready to formulate our other main result, which is a direct consequence of 
 Theorem \ref{gemstone}: 

\begin{main} \label{lodestone}
Let $M$ be the underlying smooth compact $4$-manifold of a Del Pezzo surface. Let $\mathscr{E}(M)$ denote the 
the moduli space of  Einstein metrics $h$ on $M$, and let $\mathscr{E}_\omega^+(M) \subset \mathscr{E}(M)$ be the open subset arising from   
Einstein metrics $h$ for which the corresponding  
  conformal classes $[h]$  are of  positive symplectic type. Then
$\mathscr{E}_\omega^+(M)$ is connected. 
Moreover, if $b_2(M) \leq 5$, then $\mathscr{E}_\omega^+(M)$ exactly consists of a single point. 
\end{main}
\begin{proof}
A  Del Pezzo surface is by definition a compact complex surface $(M^4, J)$ whose first Chern class is a K\"ahler class.  
As complex manifolds, the 
Del Pezzo surfaces are exactly $\CP_1 \times \CP_1$ and the blow-ups of $\CP_2$ at $m$ distinct points, $0\leq m \leq 8$, 
such that no three points are on a line, no six are on a conic, and no eight are on a nodal cubic with one of the given points at the  node \cite{delpezzo,cubic}. 
When $b_2 \leq 5$,    there is consequently, up to biholomorphism,  only one Del Pezzo complex structure for each diffeotype, since  we can simultaneously move up to 
four generically located  points in the projective plane to standard positions via a suitable projective linear transformation. 
For larger values of $b_2$, the choice of complex structure instead essentially depends on $2b_2 - 10$ complex parameters; however,  the various possibilities still 
form a single connected family, since the set of prohibited configurations of $m=b_2 - 1$ points in $\CP_2$ is a finite union of complex hypersurfaces, and so has 
  real codimension $2$.

 Given a Del Pezzo surface $(M, J)$ with  fixed complex structure, there is always a $\lambda > 0$  Einstein metric  $h$ 
 which can be written as $h=s^{-2}g$ for a  $J$-compatible  extremal K\"ahler metric $g$ with scalar curvature $s> 0$. 
  In most cases, one can simply take $h=g$, so that $h$ is a K\"ahler-Einstein metric. By a result of  
 Tian \cite{sunspot,tian},  a Del Pezzo surface $(M,J)$ admits a $J$-compatible K\"ahler-Einstein metric iff it has reductive automorphism group.
This excludes only  two  Del Pezzo surfaces, namely 
 the ones diffeomorphic to $\CP_2 \#  \overline{\CP}_2$ and $\CP_2 \# 2 \overline{\CP}_2$; 
  these two do not admit K\"ahler-Einstein metrics, but they  nonetheless admit   conformally K\"ahler, $\lambda > 0$ Einstein metrics,  known  as the Page and Chen-LeBrun-Weber metrics \cite{chenlebweb,derd,page}, respectively. 
In    \cite{lebuniq},  it was then shown 
that any conformally K\"ahler,  Einstein metric on a compact complex surface is  either K\"ahler-Einstein, or else
is isometric to a constant multiple of one of these two special metrics.
In particular, up to complex automorphisms and constant rescalings, there is exactly one conformally K\"ahler, Einstein metric
for  each Del Pezzo complex structure.  Theorem \ref{gemstone} therefore tells us that the  moduli space $\mathscr{E}_\omega^+(M)$ 
can be identified with the moduli space of Del Pezzo complex structures. Since we have seen that the latter is always  pathwise connected,
and   moreover consists of a single point when $b_2 (M) \leq 5$, the claim follows. \end{proof}

We  now conclude this introduction with a  consequence of Theorem \ref{lodestone}: 

\begin{cor}
For any Del Pezzo surface $M$, $\mathscr{E}_\omega^+(M)$ is exactly a connected component of $\mathscr{E}(M)$.
\end{cor}

Indeed, it suffices to prove that the path-connected space $\mathscr{E}_\omega^+(M)$ is both open and closed in $\mathscr{E}(M)$.
Since $\mathscr{E}(M)$ has the quotient topology, the  fact that it is open follows from the fact that the set of metrics with positive symplectic conformal class 
is  open and invariant under the action of $\Diff (M) \times \RR^\times$.   On the other hand, it is also closed, because, except in cases where $\mathscr{E}_\omega^+(M)$
is now known to be  a single point,   the Einstein metrics in question are all
K\"ahler, and   requiring that a Riemannian metric  carry a parallel almost-complex structure is a closed condition.

\section{Harmonic Self-Dual Weyl Curvature} 
\label{torso}

Recall that we say that a conformal class $[h]$ on an a compact oriented $4$-manifold $M$ is of {\em symplectic type} if there is a harmonic
self-dual $2$-form $\omega$ on $(M,h)$ such that $\omega \neq 0$ everywhere on $M$. This is indeed a conformally invariant condition, because 
the Hodge star operator is conformally invariant;  moreover,  it is an open condition \cite{lcp2} with respect to the $C^2$ topology.  
Since  any self-dual $2$-form $\omega$ satisfies 
$$\omega \wedge \omega = \omega \wedge \star \omega = |\omega |_h^2 ~d\mu_h,$$
it follows that  an appropriate $\omega$ is actually a symplectic form on $M$ if $[h]$ is of symplectic type. Assuming this, the conformally related metric $g\in [h]$ 
given by $g={2}^{-1/2}|\omega|_h h$ is then an {\em almost-K\"ahler metric}, in the sense that $g$ is related to 
 the symplectic form $\omega$ via $g= \omega (\cdot , J \cdot )$ for a unique almost-complex structure $J$ on $M$.
For our purposes, the important  point is that, in dimension $4$,  the almost-K\"ahler condition is equivalent to saying that 
 $\omega$ is harmonic and self-dual with respect to $g$, and that $|\omega|^2_g\equiv 2$.

While our primary aim here is to learn something about Einstein metrics,  we will more generally focus on oriented Riemannian $4$-manifolds 
 $(M, h)$ with {\em harmonic self-dual Weyl curvature}, in the sense that $\delta W^+:= -\nabla \cdot W^+=0$.
 When $h$ is Einstein, this property holds, as a consequence of the second Bianchi identity. However,  we will see in due course that 
 $\delta W^+=0$  is in general much weaker than the Einstein condition. 
 
 When $[h]$ is of symplectic type, it will prove  profitable to  study this equation
from the point of view of the conformally related  almost-K\"ahler metric $g$. This is quite tractable, because the divergence-free condition on
a section of $\odot^2_0 \Lambda^+$ is conformally invariant, albeit  \cite{pr2} {\em with an unexpected conformal weight}. In practice, 
this means that if $h = f^2 g$ has the property that $\delta W^+=0$, then $g$ will instead have the property that $\delta (fW^+)=0$. 
For us, the important point is that this then implies  a Weitzenb\"ock formula
\begin{equation}
\label{initio}
0 = \nabla^*\nabla (fW^+)+ \frac{s}{2} fW^+ - 6 fW^+\circ W^+ + 2 f|W^+|^2 I 
\end{equation}
for $fW^+$, considered as a section of $\End (\Lambda^+)$; cf.  \cite{derd,G1,pr2}.

To exploit this effectively, we will  need the following identity:
\begin{lem}
Any $4$-dimensional almost-K\"ahler manifold satisfies 
$$\langle W^+ , \nabla^*\nabla (\omega\otimes \omega )\rangle= [W^+(\omega , \omega )]^2 + 4 |W^+(\omega )|^2 - s W^+ (\omega , \omega )$$
at every point. 
\end{lem}
\begin{proof}First notice that the oriented Riemannian $4$-manifold $(M,g)$ satisfies
$$\Lambda^+\otimes \CC = \CC\omega \oplus K \oplus \overline{K},$$
where 
$K= \Lambda^{2,0}_J$ is the canonical line bundle of the almost-complex manifold $(M,J)$.
Locally choosing a unit section $\varphi$ of $K$, we thus have 
$$\nabla \omega = \alpha \otimes \varphi + \bar{\alpha} \otimes \bar{\varphi}$$
for a unique $1$-form $\alpha \in \Lambda^{1,0}_J$, since  $\nabla_{[a}\omega_{bc]}=0$ and $\omega^{bc}\nabla_a \omega_{bc}=  0$. 
If 
$$\circledast: \Lambda^+\times \Lambda^+\to \odot^2_0\Lambda^+$$
  denotes the symmetric trace-free product, we therefore have 
$$(\nabla_e \omega ) \circledast  (\nabla^e\omega )= 2|\alpha |^2 \varphi \circledast\bar{\varphi}  = -\frac{1}{4} |\nabla \omega|^2\omega \circledast \omega$$
and we thus deduce that 
\begin{eqnarray*} 
\langle W^+ , \nabla^*\nabla (\omega\otimes \omega )\rangle
& = &2W^+(\omega , \nabla^*\nabla \omega  ) - 2W^+(\nabla_e \omega , \nabla^e \omega ) \\
& = & 2 W^+(\omega , \nabla^*\nabla \omega  )  + \frac{1}{2}|\nabla \omega |^2 W^+(\omega , \omega )\\
& = & 2 W^+(\omega , 2W^+ ( \omega ) - \frac{s}{3} \omega  ) + 
\Big[ W^+(\omega , \omega ) -\frac{s}{3}\Big] W^+(\omega , \omega )\\
& = & -\frac{2}{3}s W^+ (\omega , \omega ) + 4 |W^+(\omega )|^2 
+ \Big[ W^+(\omega , \omega ) -\frac{s}{3}\Big] W^+(\omega , \omega )\\
& = &  [W^+(\omega , \omega )]^2 + 4 |W^+(\omega )|^2 - s W^+ (\omega , \omega )
 \end{eqnarray*}
 where we have used the Weitzenb\"ock formula 
 $$0= \nabla^* \nabla \omega - 2 W^+(\omega  ) + \frac{s}{3}\omega$$
for the harmonic self-dual $2$-form $\omega$, as well as the associated key identity 
 $$\frac{1}{2} |\nabla \omega |^2 =  W^+(\omega , \omega ) - \frac{s}{3}$$
resulting from the fact that  $|\omega |^2\equiv 2$.
 \end{proof}

Plugging this into our Weitzenb\"ock formula \eqref{initio} and integrating by parts, we thus see that whenever  a compact almost-K\"ahler $4$-manifold $(M,g, \omega)$
satisfies  $\delta (fW^+)=0$, we  then  automatically have 
\begin{eqnarray*} 0&=& 
\int_M \Big\langle\Big( \nabla^*\nabla fW^+ + \frac{s}{2} fW^+ - 6 fW^+\circ W^+ + 2 f|W^+|^2 I \Big), \omega \otimes
\omega \Big\rangle d\mu \\
&=& \int_M \Big[ \langle W^+ , \nabla^*\nabla (\omega\otimes \omega )\rangle  + \frac{s}{2} W^+(\omega , \omega ) 
 - 6 |W^+(\omega)|^2+ 2 |W^+|^2 |\omega |^2 \Big] f~ d\mu  \\
&=&  \int_M \Big[ \Big([W^+(\omega , \omega )]^2 + 4 |W^+(\omega )|^2 - s W^+ (\omega , \omega )\Big)
 \\&& \hphantom{\int_M \Big[ \Big([W^+(\omega , \omega )]^2 + 4 |W^+ } 
 + \frac{s}{2} W^+(\omega , \omega ) 
 - 6 |W^+(\omega)|^2+ 4 |W^+|^2 \Big]  f~d\mu  
 \\
&=&  \int_M \Big[ [W^+(\omega , \omega )]^2    - \frac{s}{2} W^+(\omega , \omega ) 
 - 2 |W^+(\omega)|^2+ 4 |W^+|^2 \Big]  f~d\mu ~. 
 \end{eqnarray*}
In other words, letting $W^+(\omega)^\perp$ denote the component of $W^+(\omega)$ perpendicular to $\omega$,  any compact  almost-K\"ahler manifold
$(M,g,\omega)$ with $\delta(fW^+)=0$ satisfies  the identity 
\begin{equation}
\label{gawa}
\int_M sW^+(\omega , \omega ) f~d\mu = 8\int_M \left(|W^+|^2 - \frac{1}{2} |W^+(\omega )^\perp|^2\right) f~d\mu~.
\end{equation}

To proceed further, we will now need another algebraic  observation: 
\begin{lem}
\label{dominion}
Any $4$-dimensional almost-K\"ahler manifold satisfies 
$$
|W^+|^2 - \frac{1}{2} |W^+(\omega )^\perp|^2\geq \frac{3}{8} \left[ W^+ (\omega , \omega ) \right]^2
$$
at every point, and equality can only hold at points where $W^+(\omega )^\perp=0$.
\end{lem} 
\begin{proof}
If $A=[A_{jk}]$ is any symmetric  trace-free $3\times 3$ matrix, the fact that $A_{33}= -(A_{11}+A_{22})$ implies that  
$$\sum_{jk} A_{jk}^2 \geq 2A_{21}^2 +  A_{11}^2+A_{22}^2+A_{33}^2 = 2A_{21}^2 + \frac{3}{2}A_{11}^2 + 2 (\frac{A_{11}}{2}+A_{22})^2$$
and we therefore conclude that 
$$|A|^2 \geq  2 A_{21}^2 + \frac{3}{2} A_{11}^2 .$$

If we now let $A$ represent $W^+:\Lambda^+\to \Lambda^+$ with respect to an 
orthogonal basis $\varepsilon_1, \varepsilon_2, \varepsilon_3$ for $\Lambda^+$ such that  $\omega = \sqrt{2}\varepsilon_1$ and 
$W^+(\omega )^\perp \propto \varepsilon_2$, this inequality becomes
$$
|W^+|^2 \geq |W^+(\omega)^\perp|^2 + \frac{3}{8} \left[ W^+ (\omega , \omega )\right]^2
$$
which not only proves the desired inequality, but shows that it is actually strict whenever $\omega$ is  not an eigenvector of $W^+$. 
 \end{proof}

\noindent 
Combining \eqref{gawa}  with Lemma \ref{dominion}  now yields the 
  global inequality 
\begin{equation}
\label{seki}
\int_M sW^+(\omega , \omega ) f~d\mu \geq 3 \int_M \left[ W^+ (\omega , \omega ) \right]^2 f~d\mu ,
\end{equation}
with equality only if $W^+(\omega )^\perp\equiv 0$. 
It thus  follows that 
$$
0\geq \int_M W^+(\omega , \omega ) \left( W^+ (\omega , \omega ) -\frac{s}{3}\right) ~ f~d\mu .
$$
However, since 
$\frac{1}{2} |\nabla \omega |^2= W^+(\omega , \omega ) - \frac{s}{3} $ for any almost-K\"ahler $4$-manifold, 
this proves the following: 

\begin{prop}
Let  $(M^4,g,\omega)$ be a compact almost-K\"ahler manifold, and suppose that, for some positive function $f$,  the conformally related metric $h=f^2g$
 has harmonic self-dual 
Weyl curvature. Then $(M,g,\omega)$ satisfies the inequality 
\begin{equation}
0\geq \int_M    W^+ (\omega , \omega ) |\nabla \omega |^2 f~d\mu . 
\label{punch}
\end{equation}
\end{prop}

This has  an interesting immediate consequence:

\begin{prop} 
\label{clarion}
Let  $(M^4,g,\omega)$ be a compact connected  almost-K\"ahler manifold with  
$W^+(\omega , \omega ) \geq 0$, and suppose that the conformally related metric $h=f^2 g$
satisfies $\delta W^+ =0$. Then either $g$ is a K\"ahler metric with scalar curvature $s=c/f$ for some  constant $c >  0$, or else $g$ satisfies $W^+\equiv 0$, and so is an anti-self-dual metric. \end{prop}
\begin{proof} 
Recall that  $f > 0$ by convention, and that $W^+(\omega , \omega ) \geq 0$ by assumption. Thus \eqref{punch} implies that 
$$\int_M  W^+(\omega , \omega ) |\nabla \omega |^2 f~d\mu =0,$$
so  that 
$\nabla \omega=0$ wherever $W^+(\omega , \omega )\neq 0$.
If $U\subset M$ is the open subset where $W^+(\omega , \omega )\neq 0$, the restriction of $g$ to $U$  is therefore   K\"ahler.
On the other hand,   by hypothesis, $g$ satisfies $\delta (fW^+)=0$. However, for any K\"ahler manifold of real dimension $4$, 
$W^+$ is the trace-free part of $(s/4) \omega \otimes \omega$, where the scalar curvature $s$
satisfies $s=3W^+(\omega, \omega )$. It follows that  $d[fW^+(\omega, \omega )]=0$ on $U$. By continuity, we therefore
have $d[fW^+(\omega, \omega )]=0$ on the closure $\overline{U}$ of $U$, too. On the other hand, $fW^+(\omega, \omega ) \equiv 0$ on $M-\overline{U}$, so we also have 
$d[fW^+(\omega, \omega )]=0$ on the  open set $M-\overline{U}$. Hence $d[fW^+(\omega, \omega )]=0$ on all of $M$. Since $M$ is connected, it follows that $fW^+(\omega, \omega )=c/3$ for
some non-negative constant $c\geq 0$. If $c>0$, $M=U$, and $(M,g)$ is a K\"ahler manifold, with $s=3 W^+(\omega, \omega ) = c/f$. 
On the other hand, if $c=0$, we have $W^+(\omega, \omega ) \equiv 0$, and therefore have equality in
\eqref{seki}. However, this implies that $W^+(\omega )^\perp\equiv 0$, and \eqref{gawa} therefore 
implies that $W^+\equiv 0$, as claimed. 
\end{proof}

As a special case, we therefore obtain the following key result: 

\begin{thm} \label{clarity}
Let $(M,h)$ be a compact oriented Riemannian $4$-manifold with $\delta W^+=0$. If the conformal class $[h]$ is of positive symplectic type, then 
$h=s^{-2}g$ for a unique K\"ahler metric $g$ of scalar curvature $s > 0$. Conversely, if $g$ is any K\"ahler metric 
of positive scalar curvature, the conformally related metric $h=s^{-2}g$ satisfies  $\delta W^+=0$.
\end{thm}
\begin{proof}
To say that $[h]$ is of positive symplectic type means that there is a self-dual harmonic 
$2$-form $\omega$ on $(M,h)$ such that $W^+(\omega , \omega ) > 0$ at every point of $M$.
Rescaling $h$ to make $\omega$ have constant norm $\sqrt{2}$ results in 
an almost-K\"ahler metric $\hat{g}$    such that 
$h=\hat{f}^2\hat{g}$ for some positive function $\hat{f}$.
If $h$ satisfies $\delta W^+=0$, the almost-K\"ahler metric $\hat{g}$  then satisfies  $W^+(\omega , \omega )> 0$ and 
 $\delta (\hat{f}W^+)=0$, so Proposition \ref{clarion}
then tells us that $\hat{g}$ is actually K\"ahler,  with scalar curvature $\hat{s}=c/\hat{f}$ for some positive constant $c$.  In particular,  $M$ admits a K\"ahler metric with positive scalar curvature, and Yau's vanishing 
theorem \cite{yauruled}  for the geometric genus  therefore implies that 
$b_+(M) = 1$. Thus the choice of $\omega$ is in fact unique up to an overall multiplicative constant,
and the choice 
 of $\hat{g}$ is therefore  determined up to constant rescalings. 
But if, for  a  positive constant $a$, we
 replace $\hat{g}$ with  $g=a^2\hat{g}$, we must also replace  $\hat{f}$ with $f=a^{-1}\hat{f}$; and
  note  that the scalar curvature of  $g$ is then  $s=a^{-2}\hat{s}$.
Since $\hat{f}=c\hat{s}^{-1}$, we then have $h=\hat{f}^2 \hat{g}= c^2 \hat{s}^{-2} \hat{g}= c^2 (a^{2}s)^{-2}(a^{-2}g)= (ca^{-3})^2 s^{-2}g$. This shows that 
setting $a=\sqrt[3]{c}$ results in 
a K\"ahler metric $g$ such that $h=s^{-2} g$, and moreover shows that this choice  yields the only 
K\"ahler metric with this property. 

Conversely \cite{derd}, if $g$ is a K\"ahler metric with $s> 0$, $s^{-1}W^+$ is parallel, so that, in particular,  we have $\delta (s^{-1}W^+)=0$. Thus 
$h=s^{-2}g$ satisfies $\delta W^+=0$, as promised. 
\end{proof}

Theorem \ref{gemstone} is now a straightforward consequence. Indeed, since the second Bianchi identity implies that 
any Einstein metric on an oriented $4$-manifold satisfies $\delta W^+=0$, Theorem \ref{clarify} tells us that  every 
Einstein metric $h$  of positive symplectic type must be conformally K\"ahler. Moreover, since the 
conformal class $[h]$ contains a representative $g$ with $s > 0$, the constant scalar curvature $4\lambda$  of $h$ must \cite{trud} be positive, 
too. Theorem \ref{gemstone}  therefore follows from the known classification \cite{lebuniq} of  conformally K\"ahler, Einstein   
metrics on compact $4$-manifolds. 

\section{Almost-K\"ahler Manifolds Revisited}

As an added bonus, the results of \S \ref{torso} also have interesting consequences  in 
 the narrower context of  almost-K\"ahler geometry; for related work, see \cite{bladra}. Our main such application is the following:
\begin{thm} \label{clarify}
Let $(M,g,\omega)$ be a compact almost-K\"ahler $4$-manifold with non-negative scalar curvature and harmonic self-dual Weyl tensor:
$$s\geq 0, \qquad \delta W^+=0.$$
 Then
$(M,g,\omega)$ is a constant-scalar-curvature K\"ahler manifold. \end{thm}
\begin{proof}
For any almost-K\"ahler manifold, 
$$W^+(\omega , \omega ) = \frac{s}{3}+\frac{1}{2}|\nabla \omega |^2$$
so that the hypothesis $s \geq 0$ implies $W^+(\omega, \omega ) \geq  0$.
Proposition  \ref{clarion}, with $f=1$, therefore  tells us that $(M,g)$ is K\"ahler, with scalar curvature $s=c/f=c$  for some positive constant $c$, or else that $W^+\equiv 0$. In the latter case, 
we then have $0 = 3 W^+(\omega , \omega ) \geq s \geq 0$, so $s\equiv 0$, and hence
$|\nabla \omega |^2 = 2 W^+(\omega , \omega ) - 2s/3 =0$. Thus 
$(M,g)$ is constant-scalar-curvature K\"ahler, even  in the exceptional  case. 
\end{proof}

 Conversely,  any constant-scalar-curvature K\"ahler manifold of real dimension $4$  
satisfies $\delta W^+=0$, independent of the sign of $s$. While 
 the study of  ``cscK'' (constant-scalar-curvature K\"ahler) metrics on compact complex surfaces is an  active area of ongoing research, 
many existence results   are already available \cite{arpa1,dontor,klp,leblown,rolsing,yujen}.
However, we should emphasize that 
the  non-negativity of the scalar curvature  plays a crucial role in Theorem \ref{clarify}. For example,  there exist many compact almost-K\"ahler manifolds
with $W^+\equiv 0$ which are not K\"ahler.
Indeed, such examples
 can be obtained \cite{inyoungthesis} by deforming scalar-flat K\"ahler metrics through anti-self-dual conformal classes,  and then conformally rescaling to make $|\omega|\equiv \sqrt{2}$. Examples of this type 
automatically  have $s\leq 0$, with $s<0$ on an open dense subset. 

\bigskip

Since any Einstein $4$-manifold satisfies $\delta W^+=0$, Theorem \ref{clarify} provides
a new proof of Sekigawa's  breakthrough result  \cite{seki1} on the  Goldberg conjecture: 

\begin{cor}[Sekigawa] Every compact almost-K\"ahler Einstein $4$-manifold with non-negative Einstein constant  is K\"ahler-Einstein.
\end{cor}

This fact  provided  a useful guidepost en route to the present results. 

\bigskip 

The proof of Theorem \ref{clarify} in fact still works if we merely impose the ostensibly weaker hypothesis that 
$s+ t W^+ (\omega, \omega ) \geq 0$ for some constant $t \geq 0$, since 
any such hypothesis will imply that $W^+ (\omega, \omega ) \geq 0$, with $s=0$ if equality holds. 
In particular, one reaches exactly the same conclusion if we merely assume that the so-called
star-scalar curvature
$$s^* = s +  |\nabla \omega|^2 = \frac{s}{3} + 2 W^+ (\omega , \omega )$$
is non-negative:
\begin{prop} \label{lemon}
Let $(M,g,\omega)$ be a compact almost-K\"ahler $4$-manifold with non-negative star-scalar curvature and harmonic self-dual Weyl tensor:
$$s^*\geq 0, \qquad \delta W^+=0.$$
 Then
$(M,g,\omega)$ is a constant-scalar-curvature K\"ahler manifold. \end{prop}

Kirchberg \cite{kirch} has elsewhere investigated almost-K\"ahler $4$-manifolds with harmonic Weyl tensor and positive star-scalar curvature. Since  the hypothesis $\delta W =0$
is equivalent to 
$\delta W^+= \delta W^- =0,$
and is therefore  stronger than the hypothesis $\delta W^+=0$ of Proposition \ref{lemon},  we can recover many of Kirchberg's results  
from our own. For example, we can deduce  the following clarification  of  \cite[Corollary 3.13]{kirch}:

\begin{cor} \label{ruled} 
Let $(M,g,\omega)$ be a compact almost-K\"ahler $4$-manifold with non-negative scalar curvature and harmonic  Weyl tensor:
$$s\geq 0, \qquad \delta W^+=\delta W^-=0.$$
 Then  
$(M^4,g,J)$ is either a K\"ahler-Einstein  manifold with $\lambda\geq 0$, or else  is locally symmetric, with 
universal cover $(\tilde{M}, \tilde{g})$
 isometric to the Riemannian product of two constant-curvature surfaces, where one factor is a $2$-sphere. \end{cor}
 
\begin{proof}  Theorem \ref{clarify} tells us that $(M,g,J)$ is  K\"ahler and has constant scalar curvature. But since the entire Weyl tensor is actually  assumed to be harmonic,
the second Bianchi identity also  tells us that 
$$\nabla_{[c}r_{d]b} = \nabla_a{W^a}_{bcd}+ \frac{1}{6} g_{b[c}\nabla_{d]}s = 0,$$
so that the covariant derivative $\nabla r$ of the Ricci tensor must therefore be  completely symmetric. 
Decomposing  $\otimes^3 \Lambda^1_\CC$ into $\otimes^3 ( \Lambda^{1,0} \oplus \Lambda^{0,1})$,
we thus have 
$$\nabla_\kappa r_{\mu \bar{\nu}}= \nabla_{\bar{\nu}}r_{\mu \kappa} = 0 \qquad  \mbox{ and } 
\qquad \nabla_{\bar{\kappa}} r_{\mu \bar{\nu}}= \nabla_{\mu}r_{\bar{\kappa}\bar{\nu}} = 0.$$
The Ricci tensor of our K\"ahler manifold is therefore parallel, and the primitive part $\mathring{\rho}\in \Lambda^-$ of its Ricci form must therefore be parallel, too.
If $\mathring{\rho}=0$, $(M,g,J)$ is K\"ahler-Einstein. Otherwise,  the holonomy  of $(M,g)$ fixes both $\omega$ and $\mathring{\rho}$, and so    must be  contained in 
$U(1) \times U(1)\subset U(2) \subset SO(4)$. In the latter case, the de Rham splitting theorem \cite{bes} then implies that 
 the universal cover $(\tilde{M}, \tilde{g})$ of $(M,g)$ is   a  Riemannian product  
$(M_1,g_1)\times (M_2, g_2)$ of two complete, constant-curvature Riemann surfaces; and if $g$ is not Einstein, and therefore not flat, the assumption that $s\geq 0$ then forces 
at least one factor  $(M_j,g_j)$ to have positive Gauss curvature. 
\end{proof}

The Narasimhan-Seshadri theorem \cite{narsesh} provides a complete existence theory for  the non-Einstein  metrics of Corollary \ref{ruled}.
 Indeed,  a compact complex manifold   $(M^4,J)$  admits such a locally-product  
 K\"ahler metric  iff it is   a geometrically  ruled complex surfaces that    arises  as the projectivization of a 
polystable rank-$2$ holomorphic vector bundles  over  a compact complex curve. For related results,  see \cite{chrives,burbar,lpm}.

\bigskip
\noindent
{\bf Acknowledgments.} The author would like to thank Tedi Draghici for subsequently 
pointing out  some of his  own related work, 
 and  the anonymous referee for suggesting ways to streamline and clarify  the exposition. 
This work was supported in part by   NSF grant DMS-1205953. 

\pagebreak

\vfill 

\noindent 
{\sc Department of Mathematics, State University of New York, Stony Brook, NY 11794-3651 USA} 

\medskip 

\noindent 
{e-mail:} claude@math.sunyb.edu

\bigskip 

\noindent 
{\sc Keywords:} Einstein metric, Del Pezzo surface, Weyl curvature,  moduli space, harmonic form, K\"ahler, almost-K\"ahler, symplectic.

\bigskip 

\noindent 
{\sc MSC classification:}  53C25 (Primary),  14J26,  32J15, 53C55, 53D05.

\end{document}